\documentclass[12pt]{amsart}
\usepackage[a4paper,hmargin={2.5cm,2.5cm},vmargin={2.5cm,2.5cm}]{geometry}

\usepackage[utf8]{inputenc}
\usepackage{lipsum}
\usepackage[T1]{fontenc}
\usepackage{lmodern}
\usepackage{mathrsfs}
\usepackage[T1]{fontenc}
\usepackage[english]{babel}
\usepackage{csquotes}
\usepackage{hyperref}
\usepackage{cite}
\usepackage[shortlabels]{enumitem}
\usepackage{libertinus}
\usepackage{libertinust1math}
\usepackage{mathtools}
\mathtoolsset{showonlyrefs}

\title[Multiplicity of weak solutions]{On the multiplicity of weak solutions for a class of coupled quasilinear elliptic systems}
\date{\today}

\author{Annamaria Canino}
\address{Dipartimento di Matematica e Informatica, Università della Calabria,
Arcavacata di Rende, Cosenza, Italy}
\email{annamaria.canino@unical.it}

\author{Simone Mauro}
\address{Dipartimento di Matematica e Informatica, Università della Calabria,
Arcavacata di Rende, Cosenza, Italy}
\email{simone.mauro@unical.it}

\numberwithin{equation}{section}

\usepackage[leqno]{amsmath}

\usepackage{tikz-cd}

\theoremstyle{plain}
\newtheorem{theorem}{Theorem}[section]

\newtheorem{proposition}[theorem]{Proposition}
\newtheorem{corollary}[theorem]{Corollary}

\theoremstyle{plain}
\newtheorem{definition}[theorem]{Definition}

\theoremstyle{plain}
\newtheorem{remark}[theorem]{Remark}

\RequirePackage{dsfont}

\newcommand{\R}{\field{R}}

\newcommand{\field}[1]{\mathbb{#1}}

\tikzset{%
    symbol/.style={%
        draw=none,
        every to/.append style={%
            edge node={node [sloped, allow upside down, auto=false]{$#1$}}}
    }
}
\usepackage{enumitem}

\tikzset{shorten <>/.style={shorten >=#1,shorten <=#1}}

\begin{document}

\begin{abstract}
We study the existence and regularity of weak solutions to the following quasilinear elliptic system:
\[
\begin{cases}
    -\mathrm{div}(A_k(x, u_k) |\nabla u_k|^{p_k - 2} \nabla u_k) + \dfrac{1}{p_k} D_s A_k(x, u_k) |\nabla u_k|^{p_k} = g_k(x, u) \quad \text{in } \Omega, \\
    u_k = 0 \quad \text{on } \partial\Omega, \qquad k=1,\dots,d,
\end{cases}
\]
where $\Omega \subset \mathbb{R}^N$ is a bounded domain with $N \geq 2$, $\boldsymbol{p} = (p_1, \dots, p_d)$, $p_k > 1$.
Using tools from nonsmooth critical point theory, we prove the existence of infinitely many weak solutions in $W_0^{1,\boldsymbol{p}}(\Omega) \cap L^\infty(\Omega; \mathbb{R}^d)$, where $W_0^{1,\boldsymbol p}(\Omega)=W_0^{1,p_1}(\Omega)\times\dots\times W_0^{1,p_d}(\Omega)$.


\end{abstract}

\maketitle

\section{Introduction}
Let $\Omega \subset \mathbb{R}^N$ be a bounded domain with $N \geq 2$, $p_1, \dots, p_d > 1$, $\boldsymbol{p}=(p_1,\dots,p_d)$, and  
\[
W_0^{1,\boldsymbol{p}}(\Omega) := W_0^{1,p_1}(\Omega) \times \dots \times W_0^{1,p_d}(\Omega).
\]  
We consider the functional $f: W_0^{1,\boldsymbol{p}}(\Omega) \to \mathbb{R}$, defined as follows:
\[
f(u) = \sum_{k=1}^d \frac{1}{p_k} \int_\Omega A_k(x, u_k) |\nabla u_k|^{p_k} - \int_\Omega G(x, u).
\]
For every $k=1,\dots,d$, we assume that
\begin{itemize}
\item
 $A_k: \Omega \times \mathbb{R} \to \mathbb{R}$ is a $C^1$-Carathéodory function, i.e., $A_k(\cdot, s)$ is measurable for every $s \in \mathbb{R}$, and $A_k(x, \cdot)$ is $C^1$ for a.e. $x \in \Omega$. 
 \item There exist constants $C_0, \nu > 0$ such that  
\begin{align}\tag{$a.1$}\label{a.1}
    |A_k(x, s)|, |D_s A_k(x, s)| \leq C_0,\qquad&\text{for a.e. $x\in\Omega$, for every $s\in\R$,}\\
 \tag{$a.2$}\label{a.2}
    A_k(x, s) \geq \nu > 0,\qquad&\text{for a.e. $x\in\Omega$, for every $s\in\R$,}
\end{align}  
\item $G: \Omega \times \mathbb{R}^d \to \mathbb{R}$ is a $C^1$-Carathéodory function such that $G_{s_k}(x, 0):=D_{s_k}G(x,0) = 0$ and $G(x,0)=0$,  
\item there exists $C>0$ such that
\begin{equation}\tag{$g.1$}\label{g.1}
     \left|G_{s_k}(x,s)\right|\le C\left(1+|s_k|^{q_k-1}+\sum_{j\ne k}|s_j|^{q_j\frac{q_k-1}{q_k}}\right),\quad \text{and $q_k\in(p_k,p_k^*)$},   
    \end{equation}
    
for a.e. $x \in \Omega$ and for every $s \in \mathbb{R}^d$, where 
\[
p_k^* = 
\begin{cases}
\frac{N p_k}{N - p_k}, & \text{if } N > p_k, \\  
\infty, & \text{otherwise}.
\end{cases}
\]  
\item There exist constants $R > 0$, $\mu >\overline p:= \max\{p_1,\dots,p_d\}$, and $\gamma \in (0, \mu - \overline p)$ such that, for a.e. $x\in\Omega$, for every $s\in\R^d$ and for every $k=1,\dots,d$,
\begin{align}
\tag{$g.2$}\label{g.2}
    |s| \geq R \implies 0 < \mu G(x, s) \leq \sum_{k=1}^d s_k G_{s_k}(x, s),\\
\tag{$a.3$}\label{a.3}
    |s_k|\ge R \implies s_kD_s A_k(x, s_k)\ge0,\\
\tag{$a.4$}\label{a.4}
    s_kD_s A_k(x, s_k) \leq \gamma A_k(x, s_k)\quad\text{for every }s_k\in\R.
\end{align}

\item We denote $g_k(x, s) := G_{s_k}(x, s)$.
\end{itemize}
Under these assumptions, the energy functional is continuous in $W_0^{1,\boldsymbol{p}}(\Omega)$ (as shown in Theorem \ref{differenziabilità pfunzionale} below), but it is Gâteaux differentiable only along the directions $v \in W_0^{1,\boldsymbol{p}}(\Omega) \cap L^\infty(\Omega; \mathbb{R}^d)$ and
\begin{align*}
    \langle f'(u), v \rangle = \sum_{k=1}^d \int_\Omega A_k(x, u_k) |\nabla u_k|^{p_k - 2} \nabla u_k \cdot \nabla v_k &+ \sum_{k=1}^d \frac{1}{p_k} \int_\Omega D_s A_k(x, u_k) |\nabla u_k|^{p_k}v_k \\
    &- \sum_{k=1}^d \int_\Omega g_k(x, u) v_k.
\end{align*}
Hence, a function $u \in W_0^{1,\boldsymbol{p}}(\Omega)$ such that $\langle f'(u), v \rangle = 0$ for every $v \in C_c^\infty(\Omega; \mathbb{R}^d)$ is a weak solution in $\mathcal{D}'(\Omega)$ for  
\[
\tag{$\mathcal{P}$}\label{P}
\begin{cases}
    -\mathrm{div}(A_k(x, u_k) |\nabla u_k|^{p_k - 2} \nabla u_k) + \dfrac{1}{p_k} D_s A_k(x, u_k) |\nabla u_k|^{p_k} = g_k(x, u) \quad \text{in } \Omega, \\
    u_k = 0 \quad \text{on } \partial\Omega, \qquad k=1,\dots,d.
\end{cases}
\]
Namely:  
\begin{equation}\label{weak formulation}
\int_\Omega A_k(x, u_k) |\nabla u_k|^{p_k - 2} \nabla u_k \cdot \nabla v_k + \frac{1}{p_k} \int_\Omega D_s A_k(x, u_k) |\nabla u_k|^{p_k} v_k = \int_\Omega g_k(x, u) v_k,
\end{equation} 
for every $v \in C_c^\infty(\Omega; \mathbb{R}^d)$ and $k = 1, \dots, d$.

The purpose of this paper is to prove multiplicity and boundedness results for critical points of the functional $f$ by using nonsmooth critical point theory \cite{nonsmooththeory1, degiovanni1994critical}.  Since $f$ is continuous on $W_0^{1,\boldsymbol p}(\Omega)$ but is not, in general, differentiable on the whole space, the classical variational approach cannot be applied directly.  We therefore work with the weak slope and with the corresponding nonsmooth notions of critical point, Palais--Smale sequence and compactness condition, following the line of research initiated in \cite{caninoquasilineare1, caninoserdica, nonsmooththeory1} and further developed in \cite{canino2025neumann, pellacci1997critical}.

Let us first recall some related results for scalar quasilinear equations.  The existence of infinitely many solutions to the Dirichlet problem
\[
-\operatorname{div}(A(x,u)\nabla u)+\frac12\, D_s A(x,u)\nabla u\cdot\nabla u = g(x,u)\quad \text{in }\Omega,
\qquad
u=0\quad \text{on }\partial\Omega,
\]
where $A(x,s)$ is a matrix with $C^1$-Carathéodory coefficients, was obtained in \cite{caninoquasilineare1, caninoserdica, nonsmooththeory1}.  

A recent extension of this approach to the corresponding Neumann problem has been presented in  \cite{canino2025neumann}, while a broader class of quasilinear equations of the form
 \[
-\operatorname{div}(a(x,u,\nabla u)) + b(x,u,\nabla u) = g(x,u)\quad \text{in }\Omega,
\qquad
u\in W_0^{1,p}(\Omega;\R),
\] has been investigated in \cite{pellacci1997critical}. 

 We also mention \cite{arcoyaboccardo1, candela2009some, candela2009infinitely}, where similar quasilinear problems are considered on $W_0^{1,p}(\Omega)\cap L^\infty(\Omega)$.  On this smaller space the energy functional becomes differentiable.
 However,  this restriction makes the treatment of the Palais-Smale condition more involved.

Moving from single equations to systems introduces severe analytical difficulties due to the coupling terms and the lack of smoothness. Classical variational gradient systems involving the standard Laplacian or the standard $p$-Laplacian operator, such as
\[
-\Delta u_1 = g_1(x,u_1,u_2), \qquad -\Delta u_2 = g_2(x,u_1,u_2),
\]
or
\[
    -\Delta_{p_1}u_1=g_1(x,u), \qquad
    -\Delta_{p_2}u_2=g_2(x,u), \qquad
    u=(u_1,u_2)\in W_0^{1,p_1}(\Omega)\times W_0^{1,p_2}(\Omega),
\]
have been widely investigated using smooth critical point theory, including minimization techniques, Mountain Pass arguments and linking constructions; see, for instance, \cite{deFigueiredo2000, boccardo2014some}. In this framework, the subcritical growth condition \eqref{g.1} is natural.

In the nonsmooth quasilinear setting, early multiplicity results for quadratic growth systems ($p_k=2$) were established in \cite{arioli2000existence, squassina2009existence} for problems of the form
\[
-\operatorname{div}(A_k(x,u)\nabla u_k)
   + \frac12 \sum_{i=1}^d D_{s_k}A_i(x,u)\,\nabla u_i\cdot\nabla u_i
   = g_k(x,u),
   \qquad u=(u_1,\dots,u_d)\in W_0^{1,2}(\Omega;\mathbb{R}^d).
\]
A further contribution in the $p$-growth case is given in \cite{mauro2026multiplicity}, where multiplicity and regularity are obtained for systems of the type
\[
-\operatorname{div}(A(x,u)|D u|^{p-2}\nabla u_k)
+\frac1p\nabla_{s_k}A(x,u)|D u|^p=g_k(x,u),
\qquad u\in W_0^{1,p}(\Omega;\R^d).
\]
Related gradient-type quasilinear systems have also been investigated in \cite{candela2021existence, candela2022multiple, candela2022nontrivial} by means of a different variational setting, under stronger assumptions on the growth of the nonlinearities.

The contribution of the present paper is to treat a system with possibly different exponents $p_1,\dots,p_d$ directly in the product space $W_0^{1,\boldsymbol p}(\Omega)$. Unlike the recent literature on quasilinear systems  \cite{candela2021existence, candela2022multiple, candela2022nontrivial}, we assume instead the weaker subcritical growth condition \eqref{g.1}, which corresponds to the standard one typically adopted in the regular case \cite{deFigueiredo2000} or restricted to single equations.  Thus the variational construction is not restricted to bounded functions.  The boundedness of the solutions is recovered afterwards, through a regularity argument based on \cite{Vannella23}.  Combining this regularity result with the nonsmooth Equivariant Mountain Pass Theorem, we obtain an unbounded sequence of critical levels and hence infinitely many bounded weak solutions of \eqref{P}.

Our main results are the following:  

\begin{theorem}\label{main result1}\hfill\\
Assume that hypotheses \eqref{a.1}-\eqref{a.4}, \eqref{g.1}-\eqref{g.2} are satisfied and that
\begin{equation}\tag{$a.5$}\label{a.5}
A_k(x,-s_k)=A_k(x,s_k),\qquad G(x,-s)=G(x,s),
\end{equation}
for a.e. $x\in\Omega$, for every $s\in\R^d$, and for every $k=1,\dots,d$.
Then there exists a sequence of weak solutions $\{u_h\}\subset W_0^{1,\boldsymbol{p}}(\Omega)$ of \eqref{P} such that $f(u_h)\to+\infty$ as $h\to+\infty$. Furthermore, any weak solution of \eqref{P} belongs to $W_0^{1,\boldsymbol{p}}(\Omega)\cap L^\infty(\Omega;\R^d)$.
\end{theorem}

In the special case where all exponents coincide, i.e., $p_{1}=\dots=p_{d}$, we can significantly weaken the structural assumptions (a.3) and (a.4). In recent literature (such as  \cite{candela2021existence}), a stronger super-$p$-linearity condition is typically assumed for the coefficients, imposed pointwise and for all values of each component. Here, we only require our condition (a.4p) at infinity. 

 This condition is formulated in a way that simultaneously involves all components of the system. 
For this reason, the scalar argument traditionally used in single-equation settings (e.g., \cite{caninoserdica}) cannot be applied directly because, when $d>1$, the coupling terms prevent the reduction of the problem to individual components. Despite this difficulty, we establish the following result:

\begin{theorem}\label{main result equal p}\hfill\\
Assume that $p_1=\dots=p_d=:p$, hypotheses \eqref{a.1}-\eqref{a.2}, \eqref{a.5}, 
\eqref{g.1}-\eqref{g.2} are satisfied, and that
\begin{equation}\tag{$a.4_p$}\label{a.4p}
    |s|\ge R \implies 0\le \sum_{k=1}^d s_kD_s A_k(x,s_k)|\xi_k|^p
    \le \gamma\sum_{k=1}^d A_k(x,s_k)|\xi_k|^p,
\end{equation}
for a.e. $x\in\Omega$,  and every $\xi=(\xi_1,\dots,\xi_d)\in(\R^N)^d$, and for some $\gamma\in(0,\mu-p)$. 
Then there exists a sequence of weak solutions
$\{u_h\}\subset W_0^{1,p}(\Omega;\R^d)$ of \eqref{P} such that
$f(u_h)\to+\infty$ as $h\to+\infty$. Furthermore, any weak solution of
\eqref{P} belongs to $W_0^{1,p}(\Omega;\R^d)\cap L^\infty(\Omega;\R^d)$.
\end{theorem}

\section{Preliminaries}
\subsection{Weak slope}

We recall some results on the critical point theory of continuous functionals, developed in \cite{nonsmooththeory1}. In this setting, we consider a metric space $(X,\operatorname{dist})$ and $f: X \to \R$ a continuous functional.

\begin{definition} \label{definition 1}
Let $(X,\operatorname{dist})$ be a metric space and let $f: X \to \mathds{R}$ be a continuous function. We consider $\sigma \ge 0$ such that there exist $\delta > 0$ and a continuous map $\mathscr{H}: B_{\delta}(u) \times [0, \delta] \to X$ such that
\begin{align}
\label{condition 1}
\operatorname{dist}(\mathscr{H}(v, t), v) \le t, \\
\label{condition 2}
f(\mathscr{H}(v, t)) \le f(v) - \sigma t.
\end{align}
We define
\[
|df|(u):=\sup\left\{\sigma\ge0\ :
\begin{aligned}
 &\text{there exist $\delta>0$ and}\\
  &\text{$\mathscr H\in C(B_\delta(u)\times[0,\delta];X)$}\\
&\text{satisfying \eqref{condition 1} and \eqref{condition 2}}
\end{aligned}
\right\}
\]
as the weak slope of $f$ at $u$.
\end{definition}

\begin{theorem}[{\cite[Theorem 1.1.2]{nonsmooththeory1}}]\label{theorem 1}
Let $E$ be a normed space and $X \subset E$ an open subset. Fix $u \in X$ and $v \in E$ with $\|v\| = 1$. For each $w \in X$ we define
$$\overline{D}_+f(w)[v] := \limsup_{t \to 0^+} \frac{f(w + tv) - f(w)}{t}.$$
Then $|df|(u) \ge -\limsup_{w \to u} \overline{D}_+f(w)[v]$.
\end{theorem}

\begin{definition}\label{def p.critical}
Let $X$ be a metric space and let $f: X \to \mathds{R}$ be continuous. We say that $u \in X$ is a (lower) critical point if $|df|(u) = 0$. A (lower) critical point is said to be at level $c \in \mathbb{R}$ if it is also true that $f(u) = c$.
\end{definition}

\begin{definition}\label{ps}
Let $X$ be a metric space and let $f: X \to \mathds{R}$ be continuous. A sequence $\{u_n\} \subset X$ is a $(PS)_c$-sequence if
\begin{align}
\label{ps1}
&f(u_n) \to c, \\
\label{ps2}
&|df|(u_n) \to 0.
\end{align}
Furthermore, we say that $f$ satisfies the $(PS)_c$-condition if every $(PS)_c$-sequence admits a convergent subsequence in $X$. If the $(PS)_c$-condition holds for every $c\in\R$, we will simply write $(PS)$-condition.
\end{definition}
\begin{theorem}[Equivariant Mountain Pass, {\cite[Theorem 1.3.3]{nonsmooththeory1}}]\label{MPequi}	\hfill\\
Let $X$ be a Banach space and let $f:X\to\R$ be a continuous even functional. Suppose that
\begin{itemize}
\item $\exists\ \rho>0,\alpha>f(0)$ and a subspace $W\subset X$ of finite codimension such that $f\ge\alpha$ on $\partial B_{\rho}\cap W$,
\item for every finite-dimensional subspace $V$, there exists $R=R(V)>0$ such that $f\le f(0)$ in $B_R^c\cap V$.
\end{itemize}
If $f$ satisfies the $(PS)_c$-condition for every $c\ge\alpha$, then there exists a divergent sequence of critical values, namely, there exists a sequence of critical points $\{u_n\}\subset X$ such that $c_n:=f(u_n)\to+\infty$.
\end{theorem}



\subsection{Properties of the energy functional}
In this subsection, and in the rest of the paper, $f$ is the energy functional
\[f(u)=\sum_{k=1}^d\frac{1}{p_k}\int_\Omega A_k(x,u_k)|\nabla u_k|^{p_k}-\int_\Omega G(x,u).\]
We denote the dual of the space $ W_0^{1,\boldsymbol{p}}$ by $ W^{-1,\boldsymbol{p}'}$, namely
\[ W^{-1,\boldsymbol{p}'}:=W^{-1,p_1'}\times\dots\times W^{-1,p_d'},\]
where $p_k'$ is the conjugate exponent of $p_k$ with $k=1,\dots,d$.
\begin{definition}
A sequence $\{u_n\}\subset  W_0^{1,\boldsymbol{p}}(\Omega)$ is called a Concrete Palais--Smale sequence at level $c$ for $f$, $(CPS)_c$-sequence, if
\begin{itemize}
\item $f(u_n)\to c$ in $\R$;
\item $ -\text{div}(A_k(x,u_{n,k})|\nabla u_{n,k}|^{p_k-2}\nabla u_{n,k})+\frac{1}{p_k}D_sA_k(x,u_{n,k})|\nabla u_{n,k}|^{p_k}-g_k(x,u_n) \in  W^{-1,p_k'}(\Omega)$, for $n$ large enough, for every $k=1,\dots,d$.
\item $-\text{div}(A_k(x,u_{n,k})|\nabla u_{n,k}|^{p_k-2}\nabla u_{n,k})+\frac{1}{p_k}D_sA_k(x,u_{n,k})|\nabla u_{n,k}|^{p_k}-g_k(x,u_n)\to 0$ strongly in $W^{-1, p_k'}(\Omega)$ for every $k=1,\dots,d$.
\end{itemize}
We say that $f$ satisfies the $(CPS)_c$-condition if every $(CPS)_c$-sequence admits a convergent subsequence in $ W_0^{1,\boldsymbol{p}}(\Omega)$.
\end{definition}
\begin{theorem}\label{differenziabilità pfunzionale}
Assume that hypotheses \eqref{a.1} and \eqref{g.1} hold.
The functional $f: W_0^{1,\boldsymbol{p}}(\Omega)\to\R$ is continuous and for every $u\in  W_0^{1,\boldsymbol{p}}(\Omega)$ we have
\begin{align*}
 |df|(u)\ge\sup_{\substack{\varphi\in C_c^{\infty}(\Omega; \R^d)\\ \|\varphi\|_{ W_0^{1,\boldsymbol{p}}(\Omega)}\le1}}\sum_{k=1}^d\biggl[\int_\Omega A_k(x,u_k)|\nabla u_k|^{p_k-2}\nabla u_k\cdot\nabla\varphi_k&+\frac{1}{p_k}\int_\Omega D_{s}A_k(x,u_k)|\nabla u_k|^{p_k}\varphi_k\\
 &-\int_\Omega g_k(x,u)\varphi_k\biggr].   
\end{align*}
\end{theorem}
\begin{proof}
The continuity of $f$ follows from the continuity of $A_k(x,\cdot)$ and $G(x,\cdot)$, from hypothesis \eqref{g.1}, and from Lebesgue's theorem. Let $\varphi\in C_c^\infty(\Omega;\R^d)$. Then, again by Lebesgue's theorem, the directional derivative
\begin{align*}
\langle f'(u),\varphi\rangle
&:=\lim_{t\to0}\frac{f(u+t\varphi)-f(u)}{t}\\
&=\sum_{k=1}^d\int_\Omega A_k(x,u_k)|\nabla u_k|^{p_k-2}\nabla u_k\cdot\nabla\varphi_k+\sum_{k=1}^d\frac{1}{p_k}\int_\Omega D_sA_k(x,u_k)|\nabla u_k|^{p_k}\varphi_k\\
&\quad-\sum_{k=1}^d\int_\Omega g_k(x,u)\varphi_k
\end{align*}
is well defined. By Theorem \ref{theorem 1}, applied to the direction $-\varphi$, we get, whenever $\|\varphi\|\le1$,
\[
|df|(u)\ge -\limsup_{\substack{w\to u\\ t\to0^+}}\frac{f(w-t\varphi)-f(w)}{t}=\langle f'(u),\varphi\rangle.
\]
Taking the supremum over all admissible $\varphi$ gives the desired estimate.
\end{proof}
\begin{corollary}\label{PS e CPS}
Assume that hypotheses \eqref{a.1} and \eqref{g.1} hold.
The following facts hold:
\begin{enumerate}
\item[$(i)$] if $u$ is a lower critical point for $f$, then $u$ is a weak solution of the problem \eqref{P},
\item[$(ii)$] every $(PS)_c$-sequence is also a $(CPS)_c$-sequence,
\item[$(iii)$] $f$ satisfies the $(CPS)_c$-condition $\implies f$ satisfies the $(PS)_c$-condition.
\end{enumerate}
\end{corollary}
We conclude this section with a useful Brezis-Browder type result.
This result follows from \cite{brezisbrowdergenerale}. However, for completeness, we give an explicit proof in our setting.
\begin{theorem}\label{density 2.3.2}
Assume that hypotheses \eqref{a.1} and \eqref{a.2} hold.
Let $\omega\in W^{-1,\boldsymbol{p}'}(\Omega)$ and $u\in  W_0^{1,\boldsymbol{p}}(\Omega)$ such that
\begin{equation}\label{eq thm density}
   -\text{div}(A_k(x,u_k)|\nabla u_k|^{p_k-2}\nabla u_k)+\frac{1}{p_k}D_sA_k(x,u_k)|\nabla u_k|^{p_k}=\omega_k\ \ \text{in $\mathcal D'(\Omega)$},\quad k=1,\dots,d.
\end{equation}
Let $v\in W_0^{1,\boldsymbol{p}}(\Omega)$ such that 
$\left(D_sA_k(x,u_k)|\nabla u_k|^{p_k}v_k\right)^-\in L^1(\Omega)$ for every $k=1,\dots,d$. Then, we have 
\[D_sA_k(x,u_k)|\nabla u_k|^{p_k}v_k\in L^1(\Omega)\ \  \text{and}\]
\[\int_\Omega A_k(x,u_k)|\nabla u_k|^{p_k-2}\nabla u_k\cdot\nabla v_k+\frac{1}{p_k}\int_\Omega D_{s}A_k(x,u_k)|\nabla u_k|^{p_k}v_k=\langle\omega_k,v_k\rangle,\]
for every $k=1,\dots,d$.
\end{theorem}

\begin{proof}
The case $v\in W_0^{1,\boldsymbol{p}}(\Omega)\cap L^\infty(\Omega; \R^d)$ follows by a standard density argument.

Let $\{\tilde v_n\}\subset C_c^\infty(\Omega; \R^d)$ such that $\tilde v_n\to v$ in $W_0^{1,\boldsymbol{p}}(\Omega)$ and we consider, as in \cite{brezisbrowder1}, $v_{n,k}(x)=v_k(x)\cdot\lambda_{n,k}(x)$, with $\lambda_{n,k}(x)\in[0,1]$ and
 \[v_{n,k}=v_k\bigg(|v_k|^{2}+\frac{1}{n^{2}}\bigg)^{-\frac{1}{2}}\min\bigg\{\bigg(|v_k|^{2}+\frac{1}{n^{2}}\bigg)^{\frac{1}{2}}-\frac1n,\bigg( |\tilde v_{n,k}|^{2}+\frac{1}{n^{2}}\bigg)^{\frac{1}{2}}-\frac 1n\bigg\}\in W_0^{1,p_k}(\Omega)\cap L^\infty(\Omega).\]
We point out that $v_{n,k}\to v_k$ in $W_0^{1,p_k}(\Omega)$, $|v_{n,k}|\le |v_k|$ a.e. in $\Omega$ with $k=1,\dots,d$.
Let us define 
\[T_k(x):=D_sA_k(x,u_k)|\nabla u_k|^{p_k},\qquad \text{for a.e. $x\in\Omega$.}\]
 Hence,
\[T_k(x)v_{n,k}(x)=\lambda_{n,k}(x)D_sA_k(x,u_k)|\nabla u_k|^{p_k}v_k=\lambda_{n,k}(x)T_k(x)v_k(x),\]
and $0\le\lambda_{n,k}\le1.$
Then 
\[T_kv_{n,k}=\lambda_{n,k}T_kv_k\ge-\lambda_{n,k}\left[D_sA_k(x,u_k)|\nabla u_k|^{p_k}v_k\right]^-\ge-\left[D_sA_k(x,u_k)|\nabla u_k|^{p_k}v_k\right]^-.\]
 According to Fatou's Lemma:
 \[\int_\Omega T_kv_{k}+\int_\Omega\left[D_sA_k(x,u_k)|\nabla u_k|^{p_k}v_k\right]^-\le\liminf_{n\to+\infty}\left\{\int_\Omega T_kv_{n,k}+\int_\Omega\left[D_sA_k(x,u_k)|\nabla u_k|^{p_k}v_k\right]^-\right\}.\]
 By \eqref{eq thm density} we obtain
\begin{align*}
    \int_{\Omega} D_sA_k(x,u_k)|\nabla u_k|^{p_k}v_k&\le\liminf_{n\to+\infty}\int_{\Omega}D_sA_k(x,u_k)|\nabla u_k|^{p_k}v_{n,k}\\
    &=p_k\liminf_{n\to+\infty}\bigg\{\langle\omega_k,v_{n,k}\rangle-\int_\Omega A_k(x,u_k)|\nabla u_k|^{p_k-2}\nabla u_k\cdot \nabla v_{n,k}\bigg\}\\
    &\le C_k,
\end{align*}
for some positive constant $C_k>0$.
   Thus, $D_sA_k(x,u_k)|\nabla u_k|^{p_k}v_k\in L^1(\Omega)$. Testing \eqref{eq thm density} with $v_{n,k}$, we can pass to the limit according to Lebesgue's theorem and obtain:
     \[\int_\Omega A_k(x,u_k)|\nabla u_k|^{p_k-2}\nabla u_k\cdot\nabla v_k+\frac{1}{p_k}\int_\Omega D_sA_k(x,u_k)|\nabla u_k|^{p_k}v_k=\langle\omega_k,v_k\rangle.\]
\end{proof}
\begin{remark}\label{density bounded functions}
Notice that
\[\left(D_sA_k(x,u_k)|\nabla u_k|^{p_k}\right)v_k\in L^1(\Omega),\quad k=1,\dots,d,\]
for every $v\in L^\infty(\Omega; \R^d)$ and $u\in W_0^{1,\boldsymbol p}(\Omega)$.
Hence, Theorem \ref{density 2.3.2} implies that  $v\in W_0^{1,\boldsymbol p}(\Omega)\cap L^\infty(\Omega; \R^d)$ is an admissible test function for \eqref{P}.
\end{remark}
\section{Regularity result}
Now, we investigate the $L^\infty$-regularity of the weak solutions. 

\begin{theorem}\label{regularity weak solutions}
  Assume that hypotheses \eqref{a.1}--\eqref{a.3} and \eqref{g.1} hold, and that $p_k<N$ for every $k=1,\dots,d$. Let $u=(u_1,\dots,u_d)\in W_0^{1,\boldsymbol{p}}(\Omega)$ be a weak solution of \eqref{P}.
Then $u_1,\dots,u_d\in L^\infty(\Omega)$. 
\end{theorem}

\begin{remark}\label{remaining regularity cases}
We prove Theorem \ref{regularity weak solutions} only when $p_k<N$ for every $k$, since if $p_k>N$, then $u_k\in L^\infty(\Omega)$ by the Sobolev embedding, and if $p_k=N$, we can use the embedding $W_0^{1,N}(\Omega)\hookrightarrow L^r(\Omega)$ for every finite $r$ and replace $p_k^*$ with $r_k$ large enough in the proof below.
\end{remark}

\begin{proof}
    We adapt the procedure in \cite{Vannella23}, see also \cite{carmona2013regularity} and \cite{guedda1989quasilinear}.
For every ${\tilde\gamma}>1$ and $L>1$ such that $L>R+1$, we define
\begin{equation*}
\tilde h_{L,{\tilde\gamma}}(s) =
\begin{cases}
|s|^{{\tilde\gamma}-1} s, & |s| \leq L, \\
{\tilde\gamma} L^{{\tilde\gamma}-1} s + \operatorname{sign}(s) (1-{\tilde\gamma})L^{\tilde\gamma}, & |s| > L,
\end{cases}
\end{equation*}
and $h_{L,{\tilde\gamma}}(s)$ the function such that
\begin{itemize}
    \item $h_{L,{\tilde\gamma}}\equiv0$ for $|s|\le R$; 
    \item $h_{L,{\tilde\gamma}}$ is the line which joins $(R,0)$ with the point $(R+1,(R+1)^{\tilde\gamma})$ for $s\in(R,R+1]$ and it is the line which joins $(-R,0)$ with $(-R-1,-(R+1)^{{\tilde\gamma}})$ for $s\in[-R-1,-R)$;
    \item $h_{L,{\tilde\gamma}}\equiv \tilde h_{L,{\tilde\gamma}}$ for $|s|>R+1$.
\end{itemize}
Since $h_{L,{\tilde\gamma}}$ is differentiable a.e. in $\R$ and its derivative is bounded, for every $t>1$ we can also define
\begin{equation*}
\Phi_{L,t,{\tilde\gamma}}(s) = \int_0^s|h_{L,{\tilde\gamma}}'(r)|^{\frac{t}{{\tilde\gamma}}}\, dr.
\end{equation*}
Notice that $\Phi_{L,t,{\tilde\gamma}}(s)=0$ for $|s|\le R$, while
$\frac{\Phi_{L,t,{\tilde\gamma}}(s)}{s}\ge0$ for every $s\ne0$. Hence, by \eqref{a.3},
\begin{equation}\label{segno regolarità}
D_sA_k(x,u_k)|\nabla u_k|^{p_k} \Phi_{L,t,{\tilde\gamma}}(u_k)\ge0
\end{equation}
 and, in particular,
\[
[D_sA_k(x,u_k)|\nabla u_k|^{p_k}\Phi_{L,t,\tilde\gamma}(u_k)]^-\in L^1(\Omega).
\]
According to Theorem \ref{density 2.3.2}, $\Phi_{L,t,\tilde\gamma}(u_k)$ is an admissible test function.
We test \eqref{weak formulation} with
\[
\Phi_{L,\tilde\gamma p_k,\tilde\gamma}(u_k)e_k.
\]
Here $\{e_1,\dots,e_d\}$ denotes the canonical basis of $\R^d$.  

From \eqref{segno regolarità}, we  obtain that
\[\frac{1}{p_k}\int_\Omega D_sA_k(x,u_k)|\nabla u_k|^{p_k}\Phi_{L,{\tilde\gamma} p_k,{\tilde\gamma}}(u_k)\ge0.\]
Furthermore:
\begin{equation}\label{eq test Phiuk}
\int_\Omega A_k(x,u_k)|\nabla u_k|^{p_k-2}\nabla u_k\cdot\nabla \Phi_{L,\tilde\gamma p_k,\tilde\gamma}(u_k)=\int_\Omega A_k(x,u_k)|\nabla u_k|^{p_k}|h_{L,\tilde\gamma}'(u_k)|^{p_k}.
\end{equation}
Since $h_{L,\tilde\gamma}(u_k)\in W_0^{1,p_k}(\Omega)$, the Sobolev embedding, \eqref{weak formulation} and \eqref{eq test Phiuk} imply that
\begin{align*}
    \left(\int_{\Omega}|h_{L,\tilde\gamma}(u_k)|^{q_k}\right)^{\frac{p_k}{q_k}}&\le c\int_\Omega |\nabla h_{L,\tilde\gamma}(u_k)|^{p_k}=\int_{\Omega}|\nabla u_k|^{p_k}\cdot|h_{L,\tilde\gamma}'(u_k)|^{p_k}\\
    &\le\frac{1}{\nu}\int_\Omega A_k(x,u_k)|\nabla u_k|^{p_k}|h_{L,\tilde\gamma}'(u_k)|^{p_k}\\
    &\le C_1\int_{\Omega}\left(1+|u_k|^{q_k-1}+\sum_{j\ne k}|u_j|^{q_j\frac{q_k-1}{q_k}}\right)|\Phi_{L,\tilde\gamma p_k,\tilde\gamma}(u_k)|,
\end{align*}
for some constants $c,C_1>0$ independent of $L$.
We can now follow the proof of Theorem 1.1 
 in \cite{Vannella23}.

Notice that the estimates \cite[(2.3) and (2.4)]{Vannella23} hold for $|s|\ge R+1$.

 Let $\Omega_{\sigma, w}=\{x\in\Omega\ :\ |w(x)|>\sigma\}$, for some $\sigma>R+1$. The estimate \cite[eq (2.7)]{Vannella23} becomes:
\begin{align*}
    \int_{\{|u_k|\ge R+1\}} (|u_k|^{q_k-1}+1)|\Phi_{L,\tilde\gamma p_k,\tilde\gamma}(u_k)|\le &C_2\sigma^{q_k-1}\left(\int_\Omega |h_{L,\tilde\gamma}(u_k)|^{q_k}\right)^{\frac{p_k}{q_k}\frac{\tilde\gamma p_k+1-p_k}{\tilde\gamma p_k}}\\
    &+C_2\|u_k\|_{L^{q_k}(\Omega_{\sigma,u_k})}^{q_k-p_k}\left(\int_{\Omega}|h_{L,\tilde\gamma}(u_k)|^{q_k}\right)^{\frac{p_k}{q_k}},
\end{align*}
for some $C_2>0$ independent of $L$.
We consider the coupling term:

\begin{align*}
\int_\Omega |u_j|^{q_j\frac{q_k-1}{q_k}}|\Phi_{L,\tilde\gamma p_k,\tilde\gamma}(u_k)|&\le\sigma^{q_j\frac{q_k-1}{q_k}}\int_\Omega |\Phi_{L,\tilde\gamma p_k,\tilde\gamma}(u_k)|+\int_{\Omega_{\sigma,u_j}} |u_j|^{q_j\frac{q_k-1}{q_k}}|\Phi_{L,\tilde\gamma p_k,\tilde\gamma}(u_k)|\\
&\le c_1(\sigma,R)+\sigma^{q_j\frac{q_k-1}{q_k}}\int_{\{|u_k|\ge R+1\}}|\Phi_{L,\tilde\gamma p_k,\tilde\gamma}(u_k)|\\
&\quad+\int_{\Omega_{\sigma,u_j}} |u_j|^{q_j\frac{q_k-1}{q_k}}|\Phi_{L,\tilde\gamma p_k,\tilde\gamma}(u_k)|=:c_1(\sigma,R)+I_{j,k},
\end{align*}
where $c_1(\sigma,R)>0$ is independent of $L$.

Treating the  integral $I_{j,k}$ 
 as in \cite{Vannella23}, we have:
\begin{align*}
    \sum_{j\ne k}I_{j,k}\le &\sum_{j\ne k}\Bigg\{C_3\sigma^{q_j\frac{q_k-1}{q_k}}\left(\int_\Omega|h_{L,\tilde\gamma}(u_k)|^{q_k}\right)^{\frac{p_k}{q_k}\frac{\tilde\gamma p_k+1-p_k}{\tilde\gamma p_k}}\\
    &+C_3\|u_j\|_{L^{q_j}(\Omega_{\sigma, u_j})}^{\frac{q_j}{q_k}(q_k-p_k)}\left[\left(\int_{\Omega}|h_{L,\tilde\gamma}(u_k)|^{q_k}\right)^{\frac{p_k}{q_k}}+\left(\int_{\Omega}\left|h_{L_{j,k},\tilde\gamma}(u_j)\right|^{q_j}\right)^{\frac{p_k}{q_k}}\right]\Bigg\},
\end{align*}
where $L_{j,k}:=L^{\frac{q_k}{q_j}}$ and $C_3>0$ is independent of $L$. Thus, \cite[eq (2.9)]{Vannella23} takes the form:
\begin{align*}
    \int_{\Omega}|h_{L,\tilde\gamma}(u_k)|^{q_k}&\le C_4+C_4\left(\sigma^{q_k-1}+\sum_{j\ne k}\sigma^{\frac{q_j}{q_k}(q_k-1)}\right)^{\frac{\tilde\gamma q_k}{p_k-1}}+C_4\sum_{j\ne k}\|u_j\|_{L^{q_j}(\Omega_{\sigma,u_j})}^{\frac{q_j}{p_k}(q_k-p_k)}\int_\Omega |h_{L_{j,k},\tilde\gamma}(u_j)|^{q_j},\\
\end{align*}
with $C_4>0$ independent of $L$.
Choosing now $L_k:=L^{\frac{1}{q_k}}$ for every $k=1,\dots,d$, we have
\(
L_{j,k}=L_k^{\frac{q_k}{q_j}}=L^{\frac{1}{q_j}}=L_j,
\)
and the previous estimate becomes
\begin{align*}
    \int_{\Omega}|h_{L_k,\tilde\gamma}(u_k)|^{q_k}&\le C_4+C_4\left(\sigma^{q_k-1}+\sum_{j\ne k}\sigma^{\frac{q_j}{q_k}(q_k-1)}\right)^{\frac{\tilde\gamma q_k}{p_k-1}}+C_4\sum_{j\ne k}\|u_j\|_{L^{q_j}(\Omega_{\sigma,u_j})}^{\frac{q_j}{p_k}(q_k-p_k)}\int_\Omega |h_{L_j,\tilde\gamma}(u_j)|^{q_j}.
\end{align*}
According to \cite[Lemma 2.1]{Vannella23}, we have
\[
\|u_j\|_{L^{q_j}(\Omega_{\sigma,u_j})}^{\frac{q_j}{p_k}(q_k-p_k)}\to0\qquad\text{as $\sigma\to+\infty$},\qquad \text{ for every $j=1,\dots,d$},
\]
and we can choose $\sigma\ge\sigma_1$ and $C_5>0$ independent of $L$ such that
\[\int_\Omega |h_{L_j,\tilde\gamma}(u_j)|^{q_j}\le\sum_{k=1}^d\int_\Omega |h_{L_k,\tilde\gamma}(u_k)|^{q_k}\le C_5,\qquad\text{for every $j=1,\dots,d$}.\]
Passing to the limit as $L\to+\infty$ and using the arbitrariness of $\tilde\gamma>1$, we obtain $u=(u_1,\dots,u_d)\in L^t(\Omega; \R^d)$ for every $t\in(1,\infty)$. In particular, there exists $m>\max_{1\le k\le d}\frac{N}{p_k}$ such that $g_k(x,u)\in L^m(\Omega;\R)$ for every $k=1,\dots,d$. 
{
We define 
    \[
    \xi(s)=
    \begin{cases}
        0&\text{if $|s|\le R$}\\
        \left(1-\frac{R}{|s|}\right)s&\text{if $|s|>R$}.
    \end{cases}
    \]
Thus,
\[\int_\Omega A_k(x,u_k)|\nabla u_k|^{p_k-2}\nabla u_k\cdot\nabla\xi(u_k)+\frac{1}{p_k}\int_\Omega D_sA_k(x,u_k)|\nabla u_k|^{p_k}\xi(u_k)=\int_\Omega g_k(x,u)\xi(u_k).\]
Moreover
\[\int_\Omega A_k(x,u_k)|\nabla u_k|^{p_k-2}\nabla u_k\cdot\nabla\xi(u_k)=\int_{\{|u_k|\ge R\}}A_k(x,u_k)|\nabla u_k|^{p_k-2}\nabla u_k\cdot\nabla u_k.\]
According to \eqref{a.3}, we have
\[\int_\Omega D_sA_k(x,u_k)|\nabla u_k|^{p_k}\xi(u_k)=\int_{\{|u_k|\ge R\}}\left(1-\frac{R}{|u_k|}\right)D_sA_k(x,u_k)|\nabla u_k|^{p_k}u_k\ge0.
\]

    Therefore,
    \begin{align*}
       \nu\left(\int_\Omega|\xi(u_k)|^{p_k^*}\right)^{\frac{p_k}{p_k^*}}&\le \nu C\int_{\{|u_k|>R\}}|\nabla u_k|^{p_k}\le  C\int_\Omega A_k(x,u_k)|\nabla \xi(u_k)|^{p_k}\le C\int_\Omega g_k(x,u)\xi(u_k)\\
        &\le C\left(\int_{\{|u_k|>R\}}|g_k(x,u)|^{(p_k^*)'}\right)^{\frac{1}{(p_k^*)'}}\cdot\left(\int_{\Omega}|\xi(u_k)|^{p_k^*}\right)^{\frac{1}{p_k^*}},
    \end{align*}
    where $(p_k^*)'$ is such that
    \[
    \frac{1}{p_k^*}+\frac{1}{(p_k^*)'}=1.
    \]
Hence,
\begin{align*}
\nu\left(\int_\Omega|\xi(u_k)|^{p_k^*}\right)^{\frac{p_k-1}{p_k^*}}&\le C\left(\int_{\{|u_k|>R\}}|g_k(x,u)|^{(p_k^*)'}\right)^{\frac{1}{(p_k^*)'}}\\
&\le C \left(\int_{\{|u_k|>R\}}|g_k(x,u)|^{m}\right)^{\frac{1}{m}}|\{|u_k|>R\}|^{\frac{1}{(p_k^*)'}-\frac1m}.
\end{align*}
On the other hand, for every $M>R$, we have
\begin{align*}
    \left(\int_\Omega|\xi(u_k)|^{p_k^*}\right)^{\frac{p_k-1}{p_k^*}}\ge \left(\int_{\{|u_k|>M\}}|\xi(u_k)|^{p_k^*}\right)^{\frac{p_k-1}{p_k^*}}\ge (M-R)^{p_k-1}|\{|u_k|>M\}|^{\frac{p_k-1}{p_k^*}}.
\end{align*}
Then,
\[|\{|u_k|>M\}|\le\frac{C_\nu}{(M-R)^{p_k^*}}|\{|u_k|>R\}|^{\theta},\quad\theta=\frac{p_k^*}{p_k-1}\left(\frac{1}{(p_k^*)'}-\frac{1}{m}\right)>1.\]
According to \cite[Lemma 4.1]{stampacchia1965probleme}, there exists $K>0$ such that $|\{|u_k|>M\}|=0$ for every $M\ge K$ which implies $u_k\in L^\infty(\Omega)$ for every $k=1,\dots,d$.
    
}


\end{proof}
Consequently, Theorem \ref{regularity weak solutions} and Remark \ref{density bounded functions} imply that any weak solution of \eqref{P} is an admissible test function.

\section{The Concrete Palais--Smale condition}
We recall a compactness result, which follows from \cite[Theorem 3.10]{squassina2009existence}.
\begin{theorem}\label{compactness}
    Assume that hypotheses \eqref{a.1}--\eqref{a.3} hold. Let $\{u_n\}\subset W_0^{1,\boldsymbol{p}}(\Omega)$ be a bounded sequence with $u_n=(u_{n,1},\dots,u_{n,d})$. Let $\{\omega_n\}\subset W^{-1,\boldsymbol{p}'}(\Omega)$ and $\omega\in W^{-1,\boldsymbol{p}'}(\Omega)$ be such that $\omega_n\to\omega$ in $W^{-1,\boldsymbol{p}'}(\Omega)$ and
    \[
    \int_\Omega A_k(x,u_{n,k})|\nabla u_{n,k}|^{p_k-2}\nabla u_{n,k}\cdot\nabla v_k+\frac{1}{p_k}\int_\Omega D_sA_k(x,u_{n,k})|\nabla u_{n,k}|^{p_k}v_k=\langle\omega_{n,k},v_k\rangle,
    \]
    for every $v\in C_c^\infty(\Omega; \R^d)$ and $k=1,\dots,d$.
    Then $\{u_n\}$ admits a strongly convergent subsequence in $W_0^{1,\boldsymbol{p}}(\Omega)$.
\end{theorem}
Consequently, we can state the following:
\begin{theorem}\label{CPS condition}
    Assume that hypotheses \eqref{a.1}--\eqref{a.3} and \eqref{g.1} hold. For every $c\in\R$, the following facts are equivalent:
    \begin{itemize}
        \item[$(a)$] $f$ satisfies the $(CPS)_c$-condition;
        \item[$(b)$] every $(CPS)_c$-sequence for $f$ is bounded in $W_0^{1,\boldsymbol{p}}(\Omega)$.
    \end{itemize}
\end{theorem}
\begin{proof}
   The implication $(a)\implies(b)$ is a consequence of the strong convergence. Conversely, let $\{u_n\}\subset W_0^{1,\boldsymbol{p}}(\Omega)$ be a bounded $(CPS)_c$-sequence, with $u_n=(u_{n,1},\dots,u_{n,d})$. Up to a subsequence,
   \[
   u_n\rightharpoonup u\quad\text{in }W_0^{1,\boldsymbol{p}}(\Omega),
   \]
   and, for every $k=1,\dots,d$,
   \[
   u_{n,k}\to u_k\quad\text{in }L^{t_k}(\Omega),\qquad p_k<t_k<p_k^*.
   \]
   By \eqref{g.1} and the compact Sobolev embeddings, we have
   \[
   g_k(x,u_n)\to g_k(x,u)\quad\text{in }W^{-1,p_k'}(\Omega),\qquad k=1,\dots,d.
   \]
   If
   \[
   \delta_{n,k}:=-\operatorname{div}(A_k(x,u_{n,k})|\nabla u_{n,k}|^{p_k-2}\nabla u_{n,k})+\frac{1}{p_k}D_sA_k(x,u_{n,k})|\nabla u_{n,k}|^{p_k}-g_k(x,u_n),
   \]
   then $\delta_{n,k}\to0$ in $W^{-1,p_k'}(\Omega)$ by the definition of a $(CPS)_c$-sequence. Hence $\omega_{n,k}:=g_k(x,u_n)+\delta_{n,k}$ converges to $g_k(x,u)$ in $W^{-1,p_k'}(\Omega)$. Applying Theorem \ref{compactness} with $\omega_n=(\omega_{n,1},\dots,\omega_{n,d})$, we obtain a strongly convergent subsequence.
\end{proof}

\begin{proposition}\label{prop estimate}
Assume that $p_1=\dots=p_d=:p$ and that hypotheses \eqref{a.1}-\eqref{a.2}, \eqref{a.4p}, and \eqref{g.1}   hold. Let $c\in\R$ and let
$\{u_n\}\subset W_0^{1,p}(\Omega;\R^d)$ be a $(CPS)_c$-sequence for the
functional $f$. Suppose that
\begin{equation}\label{eq positive small derivative term}
I_n:=\sum_{k=1}^d\int_{\{|u_n|\le R\}}
D_sA_k(x,u_{n,k})|\nabla u_{n,k}|^pu_{n,k}\ge0
\end{equation}
for every $n$. Let
\[
    \Omega_{n,R}=\{x\in\Omega: |u_n|\le R\},
    \qquad
    \Omega_{n,R}^c=\{x\in\Omega: |u_n|>R\}. 
\]
Then, for every $\varepsilon>0$, there exists
$K(R,\varepsilon)>0$  such that
\begin{equation}\label{eq prop estimate equal p statement}
\sum_{k=1}^d\int_{\Omega_{n,R}}A_k(x,u_{n,k})|\nabla u_{n,k}|^p
\le\varepsilon\sum_{k=1}^d\int_{\Omega_{n,R}^c}A_k(x,u_{n,k})|\nabla u_{n,k}|^p+K(R,\varepsilon).
\end{equation}
\end{proposition}
\begin{proof}
For every $k=1,\dots,d$, set
\[
\omega_{n,k}:=-\operatorname{div}\left(A_k(x,u_{n,k})|\nabla u_{n,k}|^{p-2}\nabla u_{n,k}\right)
+\frac1pD_sA_k(x,u_{n,k})|\nabla u_{n,k}|^p-g_k(x,u_n).
\]
Let $\delta\in(0,1)$ and let $\eta_\delta:[0,+\infty)\to[0,1]$ be the
Lipschitz function defined by
\[
\eta_\delta(t)=
\begin{cases}
1, & 0\le t\le R,\\[1mm]
1-\delta\log\left(\dfrac{t}{R}\right), & R<t<Re^{1/\delta},\\[3mm]
0, & t\ge Re^{1/\delta}.
\end{cases}
\]
This function  satisfies
\[
0\le\eta_\delta\le1,
\qquad
|t\eta_\delta'(t)|\le\delta\quad\text{for a.e. }t>0.
\]
We define
\[
\theta_\delta(x):=\eta_\delta(|u_n(x)|)u_n(x),
\qquad
\theta_{\delta,k}(x):=\eta_\delta(|u_n(x)|)u_{n,k}(x).
\]
Then $\theta_\delta\in W_0^{1,p}(\Omega;\R^d)\cap L^\infty(\Omega;\R^d)$ and
$|\theta_\delta|\le Re^{1/\delta}$. Moreover, writing the gradient component by
component, for a.e. $x\in\Omega$ and for every $k=1,\dots,d$ we have
\begin{equation}\label{eq gradient theta delta component}
\nabla\theta_{\delta,k}
=\eta_\delta(|u_n|)\nabla u_{n,k}
+\eta_\delta'(|u_n|)\frac{u_{n,k}}{|u_n|}
\sum_{j=1}^d u_{n,j}\nabla u_{n,j}
\end{equation}
on the set $\{|u_n|>0\}$, while the second term is understood as zero on
$\{|u_n|=0\}$. In particular, since
$|u_{n,k}|\le |u_n|$ and $|u_{n,j}|\le |u_n|$, we have
\[
|\nabla\theta_{\delta,k}|
\le |\nabla u_{n,k}|+
| |u_n|\eta_\delta'(|u_n|)|\sum_{j=1}^d|\nabla u_{n,j}|
\le |\nabla u_{n,k}|+\delta\sum_{j=1}^d|\nabla u_{n,j}|.
\]
Consequently, for a suitable constant $C>0$ independent of $n$ and $\delta$, 
\begin{equation}\label{stima gradiente}
\sum_{k=1}^d |\nabla\theta_{\delta,k}|^p
\le C\sum_{k=1}^d |\nabla u_{n,k}|^p.
\end{equation}
Since $ u_n$ solves
\[
-\operatorname{div}\left(A_k(x,u_{n,k})|\nabla u_{n,k}|^{p-2}\nabla u_{n,k}\right)
+\dfrac1pD_sA_k(x,u_{n,k})|\nabla u_{n,k}|^p
=g_k(x,u_n)+\omega_{n,k}\quad\text{in $\mathcal D'(\Omega)$}.
\]
According to Theorem \ref{density 2.3.2}, we get
\begin{equation}\label{eq theta delta test identity}
\begin{aligned}
\sum_{k=1}^d\int_\Omega A_k(x,u_{n,k})|\nabla u_{n,k}|^{p-2}\nabla u_{n,k}\cdot\nabla\theta_{\delta,k}
&+\frac1p\sum_{k=1}^d\int_\Omega D_sA_k(x,u_{n,k})|\nabla u_{n,k}|^p\theta_{\delta,k}\\
&=\sum_{k=1}^d\int_\Omega g_k(x,u_n)\theta_{\delta,k}+\sum_{k=1}^d\langle \omega_{n,k},\theta_{\delta,k}\rangle.
\end{aligned}
\end{equation}
By \eqref{eq positive small derivative term} and \eqref{a.4p}, we also deduce that
\begin{equation}\label{eq nonnegative theta delta zero order}
\frac1p\sum_{k=1}^d\int_\Omega D_sA_k(x,u_{n,k})|\nabla u_{n,k}|^p\theta_{\delta,k}\ge0.
\end{equation}

Set
\[
E_{n,\delta}:=\{R<|u_n|<Re^{1/\delta}\},
\qquad
M:=\frac{C_0}{\nu p},
\qquad
\kappa:=d\left(\frac{p-1}{p}+M\right).
\]
We estimate the first term on the left-hand side of \eqref{eq theta delta test identity}. On
$\{|u_n|\le R\}$ we have $\eta_\delta(|u_n|)=1$, hence
\[
A_k(x,u_{n,k})|\nabla u_{n,k}|^{p-2}\nabla u_{n,k}\cdot
\nabla\theta_{\delta,k}
=A_k(x,u_{n,k})|\nabla u_{n,k}|^p.
\]
On $E_{n,\delta}$, using the componentwise formula
\eqref{eq gradient theta delta component}, we get
\begin{align*}
&\sum_{k=1}^d A_k(x,u_{n,k})|\nabla u_{n,k}|^{p-2}
\nabla u_{n,k}\cdot\nabla\theta_{\delta,k}\\
&=\eta_\delta(|u_n|)
\sum_{k=1}^d A_k(x,u_{n,k})|\nabla u_{n,k}|^p\\
&\quad+\frac{\eta_\delta'(|u_n|)}{|u_n|}
\sum_{k=1}^d A_k(x,u_{n,k})u_{n,k}|\nabla u_{n,k}|^{p-2}
\nabla u_{n,k}\cdot\sum_{j=1}^d u_{n,j}\nabla u_{n,j}.
\end{align*}
Since the first term is nonnegative, and using
$\eta_\delta'(|u_n|)=-\delta/|u_n|$ on $E_{n,\delta}$,
together with
$|u_{n,k}|\le |u_n|$ and $|u_{n,j}|\le |u_n|$, we obtain
\begin{align*}
&\sum_{k=1}^d A_k(x,u_{n,k})|\nabla u_{n,k}|^{p-2}
\nabla u_{n,k}\cdot\nabla\theta_{\delta,k}\ge -\delta\sum_{k,j=1}^d
A_k(x,u_{n,k})|\nabla u_{n,k}|^{p-1}|\nabla u_{n,j}|.
\end{align*}
By Young's inequality and by \eqref{a.1}-\eqref{a.2}, on $E_{n,\delta}$ we have
\begin{align*}
&\delta\sum_{k,j=1}^d A_k(x,u_{n,k})|\nabla u_{n,k}|^{p-1}|\nabla u_{n,j}|\\
&\le \delta\sum_{k,j=1}^d A_k(x,u_{n,k})
\left(\frac{p-1}{p}|\nabla u_{n,k}|^p+
\frac1p|\nabla u_{n,j}|^p\right)\\
&\le d\delta\left(\frac{p-1}{p}+\frac{C_0}{\nu p}\right)
\sum_{k=1}^d A_k(x,u_{n,k})|\nabla u_{n,k}|^p
=\delta\kappa\sum_{k=1}^d A_k(x,u_{n,k})|\nabla u_{n,k}|^p.
\end{align*}
Consequently,
\begin{equation}\label{stima theta_delta}
\begin{aligned}
&\sum_{k=1}^d\int_\Omega A_k(x,u_{n,k})|\nabla u_{n,k}|^{p-2}
\nabla u_{n,k}\cdot\nabla\theta_{\delta,k}\\
&\quad\ge
\sum_{k=1}^d\int_{\{|u_n|\le R\}}A_k(x,u_{n,k})|\nabla u_{n,k}|^p
-\delta\kappa\sum_{k=1}^d\int_{E_{n,\delta}}A_k(x,u_{n,k})|\nabla u_{n,k}|^p.
\end{aligned}
\end{equation}
Thus, inserting \eqref{stima theta_delta} in \eqref{eq theta delta test identity} and using
\eqref{eq nonnegative theta delta zero order}, we get
\begin{equation}\label{stima theta_delta rhs}
\begin{aligned}
\sum_{k=1}^d\int_{\{|u_n|\le R\}}A_k(x,u_{n,k})|\nabla u_{n,k}|^p
&-\delta\kappa\sum_{k=1}^d\int_{E_{n,\delta}}A_k(x,u_{n,k})|\nabla u_{n,k}|^p\\
&\quad\le \sum_{k=1}^d\int_\Omega g_k(x,u_n)\theta_{\delta,k}
+\sum_{k=1}^d\langle \omega_{n,k},\theta_{\delta,k}\rangle.
\end{aligned}
\end{equation}
Now, from (weighted) Young's inequality and by the strong convergence of
$\{\omega_n\}$ in $W^{-1,p'}(\Omega;\R^d)$, we have that there exists a
constant $C_1=C_1(\delta)>0$ such that
\begin{align*}
\sum_{k=1}^d\langle \omega_{n,k},\theta_{\delta,k}\rangle
&\le C_1+\delta\|\theta_\delta\|_{W_0^{1,p}(\Omega;\R^d)}^p
=C_1+\delta\sum_{k=1}^d\int_\Omega|\nabla\theta_{\delta,k}|^p.
\end{align*}
By \eqref{stima gradiente}, we have
\begin{equation}\label{stima omega_ntheta_delta}
\begin{aligned}
\sum_{k=1}^d\langle \omega_{n,k},\theta_{\delta,k}\rangle
&\le C_1+\delta\sum_{k=1}^d\int_\Omega|\nabla\theta_{\delta,k}|^p
\le C_1+C\delta\sum_{k=1}^d\int_\Omega |\nabla u_{n,k}|^p\\
&\le C_1+C\delta\sum_{k=1}^d\int_{\{|u_n|\le R\}}|\nabla u_{n,k}|^p
+C\delta\sum_{k=1}^d\int_{\{|u_n|>R\}}|\nabla u_{n,k}|^p.
\end{aligned}\end{equation}

Taking into account \eqref{a.2}, we obtain
\[
\sum_{k=1}^d\langle \omega_{n,k},\theta_{\delta,k}\rangle
\le C_1+\frac{C\delta}{\nu}\sum_{k=1}^d\int_{\{|u_n|\le R\}}A_k(x,u_{n,k})|\nabla u_{n,k}|^p
+\frac{C\delta}{\nu}\sum_{k=1}^d\int_{\{|u_n|>R\}}A_k(x,u_{n,k})|\nabla u_{n,k}|^p.
\]
Furthermore, since $\theta_\delta=0$ on
$\{|u_n|\ge Re^{1/\delta}\}$ and
$|\theta_\delta|\le Re^{1/\delta}$, hypothesis \eqref{g.1} implies that
there exists a constant $C_2=C_2(R,\delta)>0$ such that
\begin{align*}
\sum_{k=1}^d\int_\Omega g_k(x,u_n)\theta_{\delta,k}
&\le \sum_{k=1}^d\int_\Omega |g_k(x,u_n)|\,|\theta_{\delta,k}|\\
&\le C\sum_{k=1}^d\int_{\{|u_n|\le Re^{1/\delta}\}}
\left(1+|u_{n,k}|^{q_k-1}
+\sum_{j\ne k}|u_{n,j}|^{q_j\frac{q_k-1}{q_k}}\right)|\theta_{\delta,k}|\le C_2.
\end{align*}
Hence, from \eqref{stima theta_delta rhs}--\eqref{stima omega_ntheta_delta}, we infer that there exists $C_3=C_3(R,\delta)>0$ such that
\begin{align*}
\left(1-\frac{C\delta}{\nu}\right)
\sum_{k=1}^d\int_{\{|u_n|\le R\}}A_k(x,u_{n,k})|\nabla u_{n,k}|^p
&\le \delta\left(\kappa+\frac{C}{\nu}\right)
\sum_{k=1}^d\int_{\{|u_n|>R\}}A_k(x,u_{n,k})|\nabla u_{n,k}|^p+C_3,
\end{align*}
Choosing
$\delta\in(0,1)$ sufficiently small so that
\[
0<\frac{\delta(\kappa+C/\nu)}{1-C\delta/\nu}\le\varepsilon,
\]
the conclusion follows, with $K(R,\varepsilon):=C_3/(1-C\delta/\nu)$.

\end{proof}

\begin{theorem}\label{CPS bounded equal p}
Assume that $p_1=\dots=p_d=:p$, and that hypotheses \eqref{a.1}--\eqref{a.2}, \eqref{a.4p}
and \eqref{g.1}--\eqref{g.2} hold. Let
$c\in\R$ and let
$\{u_n\}\subset W_0^{1,p}(\Omega;\R^d)$ be a $(CPS)_c$-sequence. Then
$\{u_n\}$ is bounded in $W_0^{1,p}(\Omega;\R^d)$.
\end{theorem}
\begin{proof}

We write $u_n=(u_{n,1},\dots,u_{n,d})$ and define
$\omega_n=(\omega_{n,1},\dots,\omega_{n,d})$ by
\[
\omega_{n,k}:=-\operatorname{div}(A_k(x,u_{n,k})|\nabla u_{n,k}|^{p-2}\nabla u_{n,k})+\frac1pD_sA_k(x,u_{n,k})|\nabla u_{n,k}|^p-g_k(x,u_n).
\]
Then $\omega_{n,k}\to0$ in $W^{-1,p'}(\Omega)$ for every $k=1,\dots,d$.
We argue by contradiction. Assume that $\{u_n\}$ is unbounded. Passing to a
subsequence, we may suppose that
\begin{equation}\label{eq contradiction norm infinity equal p}
\|u_n\|_{W_0^{1,p}(\Omega;\R^d)}\to+
\infty.
\end{equation}
Set
\[
I_n:=\sum_{k=1}^d\int_{\Omega_{n,R}}
D_sA_k(x,u_{n,k})|\nabla u_{n,k}|^pu_{n,k}.
\]
We will distinguish two cases, according to the sign of $I_n$.

Arguing as in the proof of Proposition \ref{prop estimate}, the negative part of
$D_sA_k(x,u_{n,k})|\nabla u_{n,k}|^pu_{n,k}$ belongs to $L^1(\Omega)$.
Therefore, by Theorem \ref{density 2.3.2}, $u_{n,k}$ is an admissible test
function and
\begin{align*}
-\|\omega_{n,k}\|_{W^{-1,p'}(\Omega)}\|u_{n,k}\|_{W_0^{1,p}(\Omega)}
&\le\langle\omega_{n,k},u_{n,k}\rangle\\
&=\int_\Omega A_k(x,u_{n,k})|\nabla u_{n,k}|^p
+\frac1p\int_\Omega D_sA_k(x,u_{n,k})|\nabla u_{n,k}|^pu_{n,k}\\
&\quad-\int_\Omega g_k(x,u_n)u_{n,k}.
\end{align*}
We first estimate the second integral in the previous formula. Since
$\Omega_{n,R}=\{x\in\Omega:|u_n|\le R\}$, by \eqref{a.1} and \eqref{a.2},
\[
\int_{\Omega_{n,R}}D_sA_k(x,u_{n,k})|\nabla u_{n,k}|^pu_{n,k}
\le C_0R\int_{\Omega_{n,R}}|\nabla u_{n,k}|^p
\le \frac{C_0R}{\nu}\int_{\Omega_{n,R}}A_k(x,u_{n,k})|\nabla u_{n,k}|^p.
\]
Moreover, by \eqref{a.4p},
\[
\sum_{k=1}^d\int_{\Omega_{n,R}^c}D_sA_k(x,u_{n,k})|\nabla u_{n,k}|^pu_{n,k}
\le \gamma\sum_{k=1}^d\int_{\Omega_{n,R}^c}A_k(x,u_{n,k})|\nabla u_{n,k}|^p.
\]
Let $\gamma'\in(\gamma,\mu-p)$ and choose $\varepsilon>0$ such that
$C_0R\varepsilon/\nu\le \gamma'-\gamma$.
In the case $I_n>0$, Proposition \ref{prop estimate} gives
\[
\sum_{k=1}^d\int_{\Omega_{n,R}}A_k(x,u_{n,k})|\nabla u_{n,k}|^p
\le\varepsilon
\sum_{k=1}^d\int_{\Omega_{n,R}^c}A_k(x,u_{n,k})|\nabla u_{n,k}|^p
+K(R,\varepsilon).
\]
 In the case $I_n\le0$ we have
\[
\sum_{k=1}^d\int_{\Omega_{n,R}}
D_sA_k(x,u_{n,k})|\nabla u_{n,k}|^pu_{n,k}\le0.
\]
Hence, in both alternatives,
\begin{equation}\label{eq derivative term prop boundedness}
\begin{aligned}
&\sum_{k=1}^d\int_\Omega D_sA_k(x,u_{n,k})|\nabla u_{n,k}|^pu_{n,k}\\
&\quad\le \left(\frac{C_0R\varepsilon}{\nu}+\gamma\right)
\sum_{k=1}^d\int_{\Omega_{n,R}^c}A_k(x,u_{n,k})|\nabla u_{n,k}|^p+K(R,\varepsilon)\\
&\quad\le \gamma'\sum_{k=1}^d\int_\Omega A_k(x,u_{n,k})|\nabla u_{n,k}|^p+K(R,\varepsilon).
\end{aligned}
\end{equation}
Since $\{u_n\}$ is a $(CPS)_c$-sequence, there exists $C>0$ such that
\begin{align*}
C+o(1)\sum_{k=1}^d\|u_{n,k}\|_{W_0^{1,p}(\Omega)}
&\ge f(u_n)-\frac1\mu\sum_{k=1}^d\langle\omega_{n,k},u_{n,k}\rangle\\
&=\sum_{k=1}^d\frac{\mu-p}{p\mu}\int_\Omega A_k(x,u_{n,k})|\nabla u_{n,k}|^p\\
&\quad-\sum_{k=1}^d\frac1{p\mu}\int_\Omega D_sA_k(x,u_{n,k})|\nabla u_{n,k}|^pu_{n,k}\\
&\quad+\int_\Omega\left[\frac1\mu\sum_{k=1}^dg_k(x,u_n)u_{n,k}-G(x,u_n)\right].
\end{align*}
By \eqref{g.2}, the last integrand is nonnegative on $\Omega_{n,R}^c$; on
$\Omega_{n,R}$ it is bounded from below by a constant depending only on $R$.
Using \eqref{eq derivative term prop boundedness}, we infer that
\[
C+o(1)\sum_{k=1}^d\|u_{n,k}\|_{W_0^{1,p}(\Omega)}
\ge \sum_{k=1}^d\frac{\mu-p-\gamma'}{p\mu}
\int_\Omega A_k(x,u_{n,k})|\nabla u_{n,k}|^p-K(R,\varepsilon).
\]
Since $\gamma'<\mu-p$ and $A_k(x,s)\ge\nu$, there is $C_1>0$ such that
\[
C+o(1)\sum_{k=1}^d\|u_{n,k}\|_{W_0^{1,p}(\Omega)}
\ge C_1\sum_{k=1}^d\|u_{n,k}\|_{W_0^{1,p}(\Omega)}^p-K(R,\varepsilon).
\]
As $p>1$, this implies the boundedness of $\{u_n\}$ in
$W_0^{1,p}(\Omega;\R^d)$, which contradicts
\eqref{eq contradiction norm infinity equal p}.
\end{proof}
\begin{theorem}\label{CPS bounded}
  Assume that hypotheses \eqref{a.1}--\eqref{a.4} and \eqref{g.1}--\eqref{g.2} hold. Let $c\in\R$ be a real number and let $\{u_n\}\subset W_0^{1,\boldsymbol{p}}(\Omega)$ be a $(CPS)_c$-sequence. Then $\{u_n\}$ is bounded in $W_0^{1,\boldsymbol{p}}(\Omega)$.
\end{theorem}
\begin{proof}
    We write $u_n=(u_{n,1},\dots,u_{n,d})$ and we define $\omega_n=(\omega_{n,1},\dots,\omega_{n,d})$ as follows:
    \[
    \omega_{n,k}:=-\operatorname{div}(A_k(x,u_{n,k})|\nabla u_{n,k}|^{p_k-2}\nabla u_{n,k})+\frac{1}{p_k}D_sA_k(x,u_{n,k})|\nabla u_{n,k}|^{p_k}-g_k(x,u_n).
    \]
    We first justify that $u_{n,k}$ is an admissible test function. On $\{|u_{n,k}|\ge R\}$, assumption \eqref{a.3} gives
    \[
    u_{n,k}D_sA_k(x,u_{n,k})\ge0,
    \]
    while on $\{|u_{n,k}|<R\}$, by \eqref{a.1},
    \[
    \left|D_sA_k(x,u_{n,k})u_{n,k}|\nabla u_{n,k}|^{p_k}\right|
    \le C_0R|\nabla u_{n,k}|^{p_k}\in L^1(\Omega).
    \]
    Therefore,
    \[
    \left(D_sA_k(x,u_{n,k})|\nabla u_{n,k}|^{p_k}u_{n,k}\right)^-\in L^1(\Omega),
    \]
    and Theorem \ref{density 2.3.2} can be applied with $v_k=u_{n,k}$. Hence
    \begin{align*}
        -\|\omega_{n,k}\|_{W^{-1,p_k'}}\cdot\|u_{n,k}\|_{W_0^{1,p_k}}&\le\langle\omega_{n,k},u_{n,k}\rangle\\
        &=\int_\Omega A_k(x,u_{n,k})|\nabla u_{n,k}|^{p_k}+\frac{1}{p_k}\int_\Omega D_sA_k(x,u_{n,k})|\nabla u_{n,k}|^{p_k}u_{n,k}\\
        &\quad-\int_\Omega g_k(x,u_n)u_{n,k}.
    \end{align*}
    By \eqref{a.4}, we have
    \begin{equation}\label{eq global derivative estimate}
        \int_\Omega D_sA_k(x,u_{n,k})|\nabla u_{n,k}|^{p_k}u_{n,k}
        \le\gamma\int_\Omega A_k(x,u_{n,k})|\nabla u_{n,k}|^{p_k},
        \qquad k=1,\dots,d.
    \end{equation}
    Since $\{u_n\}$ is a $(CPS)_c$-sequence, there exists $C_1>0$ such that
 \begin{align*}
     C_1&+\frac1\mu\sum_{k=1}^d\|\omega_{n,k}\|_{W^{-1,p_k'}(\Omega)}\|u_{n,k}\|_{W_0^{1,p_k}(\Omega)}\ge f(u_n)-\frac{1}{\mu}\langle\omega_n,u_n\rangle\\
     &\quad=f(u_n)-\frac1\mu\sum_{k=1}^d\langle\omega_{n,k},u_{n,k}\rangle=\sum_{k=1}^d\frac{\mu-p_k}{p_k\mu}\int_\Omega A_k(x,u_{n,k})|\nabla u_{n,k}|^{p_k}\\
     & -\sum_{k=1}^d \frac{1}{p_k\mu}\int_\Omega D_sA_k(x,u_{n,k})|\nabla u_{n,k}|^{p_k}u_{n,k}+\int_\Omega\left[\frac{1}{\mu}\sum_{k=1}^dg_k(x,u_n)u_{n,k}-G(x,u_n)\right].
 \end{align*}
By \eqref{g.2}, the integrand in the last integral is nonnegative on $\{|u_n|\ge R\}$. On $\{|u_n|<R\}$, it is bounded from below by a constant depending only on $R$, $G$, the functions $g_k$, and $|\Omega|$. Therefore, there exists $C_2>0$ such that
\[
\int_\Omega\left[\frac{1}{\mu}\sum_{k=1}^dg_k(x,u_n)u_{n,k}-G(x,u_n)\right]\ge -C_2.
\]
Using this estimate and \eqref{eq global derivative estimate}, we get
 \begin{align*}
     C_1+\frac1\mu\sum_{k=1}^d\|\omega_{n,k}\|_{W^{-1,p_k'}(\Omega)}\|u_{n,k}\|_{W_0^{1,p_k}(\Omega)}
     &\ge\sum_{k=1}^d\frac{\mu-p_k-\gamma}{p_k\mu}\int_\Omega A_k(x,u_{n,k})|\nabla u_{n,k}|^{p_k}-C_2\\
     &\ge\nu\sum_{k=1}^d\frac{\mu-p_k-\gamma}{p_k\mu}\|u_{n,k}\|_{W_0^{1,p_k}}^{p_k}-C_2.
 \end{align*}
Since $\gamma<\mu-\overline p$, there exists $c_0>0$ such that
\[
\frac{\mu-p_k-\gamma}{p_k\mu}\ge c_0>0\qquad\text{for every }k=1,\dots,d.
\]
Moreover, $\|\omega_{n,k}\|_{W^{-1,p_k'}(\Omega)}\to0$ for every $k=1,\dots,d$. Thus,
\[
c_0\nu\sum_{k=1}^d\|u_{n,k}\|_{W_0^{1,p_k}}^{p_k}
\le C+o(1)\sum_{k=1}^d\|u_{n,k}\|_{W_0^{1,p_k}}.
\]
Since $p_k>1$ for every $k=1,\dots,d$, this implies that $\{u_n\}$ is bounded in $W_0^{1,\boldsymbol p}(\Omega)$, and the proof is complete.
\end{proof}

\section{Main results}
\begin{proof}[Proof of Theorem \ref{main result1}]
We now proceed to verify the assumptions of Theorem \ref{MPequi}.\\
    \textbf{Step 1}: There exist a finite-codimensional subspace $\mathcal W$ and constants $\rho,\alpha>0$ such that $f(u)\ge\alpha$ for every $u\in \partial B_\rho\cap\mathcal W$.
    
    According to \eqref{g.1}, and by the mean value theorem there exists $\theta\in(0,1)$ such that
    \begin{align*}
    |G(x,s)|\le|\nabla G(x,\theta s)\cdot s|&\le C\sum_{k=1}^d\left(1+\theta^{q_k-1}|s_k|^{q_k-1}+\sum_{j\ne k}\theta^{q_j\frac{q_k-1}{q_k}}|s_j|^{q_j\frac{q_k-1}{q_k}}\right)|s_k|\\
    &\le C\sum_{k=1}^d\left(|s_k|+|s_k|^{q_k}+\sum_{j\ne k}|s_j|^{q_j\frac{q_k-1}{q_k}}|s_k|\right).
    \end{align*}
    We set $\sigma_{j,k}=q_j\frac{q_k-1}{q_k}$. Using Young's inequality, we obtain that 
    \begin{align*}
   |s_j|^{q_j\frac{q_k-1}{q_k}}|s_k|\le \kappa_1|s_j|^{q_j}+\kappa_2|s_k|^{q_k},\qquad |s_k|\le \kappa_3+\kappa_4|s_k|^{q_k},
    \end{align*}
 with $\kappa_1,\kappa_2,\kappa_3,\kappa_4>0$,   and 
    \[
    |G(x,s)|\le\kappa\sum_{k=1}^d\left(1+|s_k|^{q_k}\right),
    \]
    for some $\kappa>0$.
    
       Therefore,
    \begin{align*}
        f(u)&\ge\sum_{k=1}^d\frac{\nu}{p_k}\int_\Omega |\nabla u_k|^{p_k}-\sum_{k=1}^d\kappa\int_{\Omega}\left(1+|u_k|^{q_k}\right)\\
        &\ge\sum_{k=1}^d\frac{\nu}{p_k}\int_\Omega |\nabla u_k|^{p_k}-\kappa|\Omega|-\kappa\sum_{k=1}^d\int_{\Omega}|u_k|^{q_k}.
    \end{align*}
%

\noindent
Moreover, for every $k=1,\dots,d$, by H\"older's inequality, we also obtain
\begin{align*}
    \int_\Omega|u_k|^{q_k}\le\left(\int_\Omega|u_k|^{p_k}\right)^{\frac{\tau_k}{p_k}}\cdot\left(\int_\Omega|u_k|^{p_k^*}\right)^{\frac{q_k-\tau_k}{p_k^*}},\quad q_k<p_k^*,\ \tau_k=\frac{1-\frac{q_k}{p_k^*}}{\frac{1}{p_k}-\frac{1}{p_k^*}}.
\end{align*}
If $N\le p_k$, it is enough to replace $p_k^*$ in the previous interpolation by any exponent $r_k>q_k$ sufficiently large, using the embedding $W_0^{1,p_k}(\Omega)\hookrightarrow L^{r_k}(\Omega)$.
Thus,
\begin{align}\label{5.1}
 f(u)&\ge\sum_{k=1}^d\frac{\nu}{p_k}\int_\Omega|\nabla u_k|^{p_k}-\kappa|\Omega|-\kappa\sum_{k=1}^d\left(\int_\Omega|u_k|^{p_k}\right)^{\frac{\tau_k}{p_k}}\cdot\left(\int_\Omega|u_k|^{p_k^*}\right)^{\frac{q_k-\tau_k}{p_k^*}}.
\end{align}
According to \cite[Section 5]{candela2009infinitely} and \cite[Proposition 5.4]{candela2009some}, for every $k=1,\dots,d$ there are two sequences
$\{\lambda_{h,p_k}\}_{h\in\mathbb N}\subset(0,+\infty)$ and
$\{\varphi_{h,p_k}\}_{h\in\mathbb N}\subset W_0^{1,p_k}(\Omega)$ such that
\[
\begin{cases}
-\Delta_{p_k}\varphi_{h,p_k}=\lambda_{h,p_k}|\varphi_{h,p_k}|^{p_k-2}\varphi_{h,p_k}&\text{in $\Omega$},\\
\varphi_{h,p_k}=0&\text{on $\partial\Omega$},
\end{cases}
\]
and, for every $h_0\in\mathbb N$, a closed subspace $\mathcal W_{h_0,p_k}\subset W_0^{1,p_k}(\Omega)$ with finite codimension such that
\begin{equation}\label{dis var}
\lambda_{h_0,p_k}\int_\Omega |u_k|^{p_k}
\le
\int_\Omega |\nabla u_k|^{p_k},
\qquad
\text{for every }u_k\in\mathcal W_{h_0,p_k}.
\end{equation}
Moreover, $\lambda_{h,p_k}\to+\infty$ as $h\to+\infty$ for every $k=1,\dots,d$. We set
\[
\mathcal W_{h_0}:=
\mathcal W_{h_0,p_1}\times\dots\times \mathcal W_{h_0,p_d},
\]
which is a closed finite-codimensional subspace of $W_0^{1,\boldsymbol p}(\Omega)$.

Hence, for every $u \in \mathcal W_{h_0}$, \eqref{5.1}-\eqref{dis var} and the Sobolev embedding yield
    \begin{align*}
    f(u)&\ge\sum_{k=1}^d\frac{\nu}{p_k}\int_\Omega|\nabla u_k|^{p_k}-\kappa\sum_{k=1}^d\left(\int_\Omega|u_k|^{p_k}\right)^{\frac{\tau_k}{p_k}}\cdot\left(\int_\Omega|u_k|^{p_k^*}\right)^{\frac{q_k-\tau_k}{p_k^*}}-\kappa|\Omega|\\
    &\ge\frac{\nu}{\max\{p_1,\dots,p_d\}}\sum_{k=1}^d\|u_k\|_{W_0^{1,p_k}}^{p_k}-\kappa\sum_{k=1}^d\lambda_{h_0,p_k}^{-\frac{\tau_k}{p_k}}\|u_k\|_{W_0^{1,p_k}}^{q_k}-\kappa|\Omega|.
    \end{align*}
  Using the elementary inequality
\[
\left(\sum_{k=1}^d a_k\right)^{\gamma_0} \le C_d \sum_{k=1}^d a_k^{\gamma_0}, \qquad a_k \ge 0,\ \gamma_0 \ge 1,
\]
where one can take $C_d = d^{\gamma_0 - 1}$, we obtain in particular
\[
\left(\sum_{k=1}^d a_k\right)^{\tilde p} \le C_d \sum_{k=1}^d a_k^{\tilde p},\qquad \tilde p:=\min\{p_1,\dots,p_d\}.
\]
Moreover, by Young's inequality,
\[
a_k^{\tilde p} \le \frac{\tilde p}{p_k}a_k^{p_k} + \frac{p_k - \tilde p}{p_k} \le a_k^{p_k} + 1,
\]
for every $k=1,\dots,d$. Therefore,
\[
\left(\sum_{k=1}^d a_k\right)^{\tilde p} \le C_d \sum_{k=1}^d a_k^{p_k} + C_d d.
\]
This implies that there exist positive constants $c_1, c_2 > 0$ (depending only on $d$ and $\tilde p$) such that
\[
\sum_{k=1}^d a_k^{p_k} \ge c_1 \left(\sum_{k=1}^d a_k\right)^{\tilde p} - c_2.
\]
Hence, setting $a_k = \|u_k\|_{W_0^{1,p_k}}$, and $\overline p:=\max\{p_1,\dots,p_d\}$,  we conclude that
  \begin{align*}
    f(u)&\ge \frac{\nu}{\overline p} c_1\left(\sum_{k=1}^d \|u_k\|_{W_0^{1,p_k}}\right)^{\tilde p}
 - \left(\kappa |\Omega| + \frac{\nu}{\overline p} c_2\right)
 -\kappa\sum_{k=1}^d\lambda_{h_0,p_k}^{-\frac{\tau_k}{p_k}}\|u_k\|_{W_0^{1,p_k}}^{q_k}.
    \end{align*}
Then $f(u)\ge\alpha>0$ on $\mathcal W_{h_0}\cap\partial B_\rho$ for a suitable choice of $\rho,\alpha>0$ and $h_0\in\mathbb N$ large enough. Indeed, for every $u\in\partial B_\rho$, there exists $\tilde \kappa>0$ such that
  \[
  f(u)\ge\frac{\nu}{\overline p}c_1\rho^{\tilde p}-\tilde\kappa-\kappa\sum_{k=1}^d\lambda_{h_0,p_k}^{-\frac{\tau_k}{p_k}}\rho^{q_k}\ge1-\kappa\sum_{k=1}^d\lambda_{h_0,p_k}^{-\frac{\tau_k}{p_k}}\rho^{q_k},
  \]
  for a suitable choice of $\rho>0$. Since $\lambda_{h_0,p_k}\to+\infty$ as $h_0\to+\infty$ for every $k$, choosing $h_0$ large enough gives $f(u)\ge\alpha$ for some $\alpha>0$.\\
    \textbf{Step 2:} For every finite-dimensional subspace $V\subset W_0^{1,\boldsymbol{p}}(\Omega)$ there exists $R>\rho$ such that $f\le 0$ in $B_R^c\cap V$.
    
    Let $\varphi\in V$ with $\|\varphi\|=1$. By \eqref{g.2} we have that 
    \[G(x,s)\ge b_0(x)|s|^{\mu}-a_0(x),\ \ \text{with $b_0\in L^1(\Omega)$, $b_0>0$ a.e., and $a_0\in L^1(\Omega)$},\]
  see, for instance, \cite[Proof of Theorem 3.1, page 20]{squassina2009existence}.
   Therefore,
    \begin{align}\label{geometria 2}
        f(t\varphi)\le \sum_{k=1}^d\frac{Ct^{p_k}}{p_k}-t^{\mu}\int_\Omega b_0(x)|\varphi|^{\mu}+\|a_0\|_{L^1}\to-\infty,
    \end{align}
    as $t\to+\infty$. Therefore, the claim follows by taking $R>0$ large enough.\\
\textbf{Step 3}: The functional $f$ satisfies the $(PS)$-condition.

Let $c\in\R$ and let $\{u_n\}\subset W_0^{1,\boldsymbol p}(\Omega)$ be a $(PS)$-sequence at level $c$. By Corollary \ref{PS e CPS}-$(ii)$, $\{u_n\}$ is also a $(CPS)$-sequence and Theorem \ref{CPS bounded} and Theorem \ref{CPS condition} imply that the $(CPS)$-condition holds. Corollary \ref{PS e CPS}-$(iii)$ allows us to conclude that the functional also satisfies the $(PS)$-condition.
    
    Now, by Theorem \ref{MPequi} we obtain a sequence of critical points $\{u_n\}\subset W_0^{1,\boldsymbol p}(\Omega)$ for $f$ such that $f(u_n)\to+\infty$, as $n\to+\infty$. By Theorem \ref{regularity weak solutions}, together with Remark \ref{remaining regularity cases} for the cases $p_k\ge N$, each $u_n$ belongs to $W_0^{1,\boldsymbol p}(\Omega)\cap L^\infty(\Omega;\R^d)$. This completes the proof.
\end{proof}

\begin{proof}[Proof of Theorem \ref{main result equal p}]
The proof follows the same argument as the proof of Theorem \ref{main result1}.

Steps 1 and 2 are obtained  taking into account that
$p_1=\dots=p_d=p$ and using hypothesis \eqref{a.4p} in place of
\eqref{a.3} and \eqref{a.4}.

Step 3 is obtained using Theorem \ref{CPS bounded equal p} instead of Theorem \ref{CPS bounded}.


Therefore, Theorem \ref{MPequi} yields a sequence of critical points
$\{u_n\}\subset W_0^{1,p}(\Omega;\R^d)$ such that
$f(u_n)\to+\infty$ as $n\to+\infty$. Finally, by Theorem
\ref{regularity weak solutions}, together with Remark
\ref{remaining regularity cases} for the cases $p\ge N$, each $u_n$ belongs
to $W_0^{1,p}(\Omega;\R^d)\cap L^\infty(\Omega;\R^d)$. This completes the
proof.
\end{proof}

\end{document}